\documentclass[letterpaper, 10 pt, conference]{ieeeconf}  

\IEEEoverridecommandlockouts                              

\usepackage{graphics} 
\usepackage{epsfig} 
\usepackage{mathptmx} 
\usepackage{times} 
\usepackage[mathscr]{euscript}
\usepackage{amsmath} 
{
      \newtheorem{assumption}{Assumption}
      \newtheorem{theorem}{\bf{Theorem}}[section]
      \newtheorem{definition}{\bf{Definition}}
       
      \newtheorem{lemma}[theorem]{\bf{Lemma}}
      \newtheorem{remark}{\bf{Remark}}

}
\usepackage{amssymb}  
\usepackage{mathtools}

\usepackage[ruled,vlined]{algorithm2e}
\usepackage[backend=biber,
style=ieee,
citestyle=numeric 
]{biblatex}
\DeclareMathOperator*{\argmin}{arg\,min}
\DeclareMathOperator*{\Minimize}{\textsf{Minimize}}

\title{\LARGE \bf
Attack-Resilient Weighted $\ell_1$ Observer with Prior Pruning
}

\author{Yu Zheng$^{1}$ and Olugbenga Moses Anubi$^{1}$

\thanks{$^{1}$Yu Zheng and Olugbenga Moses Anubi are with the Department of Electrical and Computer Engineering, Center for Advanced Power systems (CAPS), Center for Intelligent Systems Control and Robotics (CISCOR), Florida State University, Tallahassee, FL 32310, USA (e-mail: yz19b@fsu.edu, oanubi@fsu.edu)}
}

\begin{document}

\maketitle
\thispagestyle{empty}
\pagestyle{empty}

\begin{abstract}
 Security related questions for Cyber Physical Systems (CPS) have attracted much research attention in searching for novel methods for attack-resilient control and/or estimation. Specifically, false data injection attacks (FDIAs) have been shown to be capable of bypassing bad data detection (BDD), while arbitrarily compromising the integrity of state estimators and robust controller even with very sparse measurements corruption. Moreover, based on the inherent sparsity of pragmatic attack signals, \emph{$\ell_1$-minimization} scheme has been used extensively to improve the  design of attack-resilient estimators. For this, the theoretical maximum for the percentage of compromised nodes that can be accommodated has been shown to be $50\%$. In order to guarantee correct state recoveries for larger percentage of attacked nodes, researchers have begun to incorporate \emph{prior} information into the underlying resilient observer design framework. For the most pragmatic cases, this \emph{prior} information is often obtained through some data-driven machine learning process. Existing results have shown strong positive correlation between the tolerated attack percentages and the precision of the \emph{prior} information. In this paper, we present a \emph{pruning} method to improve the precision of the \emph{prior} information, given corresponding stochastic uncertainty characteristics of the underlying machine learning model. Then a \emph{weighted $\ell_1$-minimization} is proposed based on the \emph{pruned prior}. The theoretical and simulation results show that the pruning method significantly improves the observer performance for much larger attack percentages, even when moderately accurate machine learning model used. 

\end{abstract}

\section*{NOTATION} \label{Sec: Notation}
The following notations and definitions are used throughout the whole paper: ${\mathbb R}, {\mathbb R}^n, {\mathbb R}^{n \times m}$ denote the space of real numbers, real vectors of length $n$ and real matrices of $n$ rows and $m$ columns respectively. ${\mathbb R}_{+}$ denotes positive real numbers. Normal-face lower-case letters $(e.g.\hspace{1mm} x \in {\mathbb R})$ are used to represent real scalars, bold-face lower-case letter $(e.g.\hspace{1mm} \mathbf{x} \in {\mathbb R}^n)$ represents vectors, while normal-face upper case $(e.g.\hspace{1mm} X \in {\mathbb R}^{n \times m})$ represents matrices. Let $\mathscr{T} \subseteq \{1,\dots, n\}$, then for a matrix $X \in {\mathbb R}^{m \times n}$, $X_{\mathscr{T}} \in {\mathbb R}^{|\mathscr{T}| \times n}$ is the sub-matrix obtained by extracting the rows of $X$ corresponding to the indices in $\mathscr{T}$. $\mathscr{T}^c$ denotes the complement of a set $\mathscr{T}$ and the universal set on which it is defined will be clear from the context. The support of a vector $\mathbf{x}$ is denoted by $\textsf{supp}(\mathbf{x}) \triangleq \{i|\mathbf{x}_{i} \neq 0\}$.  The set $\mathscr{S}_k^m \subset \mathbb{R}^n$ denotes the set of all vectors $\mathbf{v} \in \mathbb{R}^m$ such that $|\textsf{supp}(\mathbf{v})|\leq k$ (i.e the subset of $k$-sparse vectors). The \emph{best $k$-sparse approximation error} for $\mathbf{e} \in \mathscr{S}^m_{k}$ is given by
\begin{equation}  \label{equ:best_k_sparse_estimation_error}
    \sigma_k(\mathbf{e}) \triangleq \min_{\mathbf{z} \in \mathscr{S}^n_k} \|\mathbf{e}-\mathbf{z}\|_1 =\|\mathbf{e}_{\mathscr{T}^c}\|_1,
\end{equation}
where $\mathscr{T}$ is the support of $\mathbf{e}_i$ with first $k$ largest magnitude. The symbol $*$ denotes the convolution operator for vectors. The symbol $\circ$ denotes element-wise multiplication of two vectors and is defined as $\mathbf{z} = \mathbf{x} \circ \mathbf{y}$, where $\mathbf{z}_{i} = \mathbf{x}_{i} \cdot \mathbf{y}_{i}$. $\mathscr{B}(1,\mathbf{p}_i)$ denotes Bernoulli distributed variables with known independent probability $\mathbf{p}_i$.

\if false
and the probability mass function of sum of these variables is given in next Lemma.
\begin{lemma}(\cite{fernandez2010closed})  \label{Lem:PBD_CDF}
Given mutually independent Bernoulli random variables $\epsilon_i \sim \mathscr{B}(1,\mathbf{p}_i)$, $i=1,\cdots,N$, then
  \begin{equation}\label{equ:PMF}
      \textsf{Pr}\left\{\sum_{i=1}^{N} \epsilon_{i} = k-1 \right\}=\mathbf{r}(k), k=1,\cdots,N+1
  \end{equation}
  where,
  $$ \mathbf{r} = \beta \cdot \begin{bmatrix}-\mathbf{s}_1\\1 \end{bmatrix} * \begin{bmatrix}-\mathbf{s}_2\\1 \end{bmatrix} * \cdots * \begin{bmatrix}-\mathbf{s}_{N}\\1 \end{bmatrix}
  $$
  with
  $ \beta =\displaystyle\prod_{i=1}^{N} \mathbf{p}_{i}, \hspace{2mm} \mathbf{s}_i = -\frac{1-\mathbf{p}_{i}}{\mathbf{p}_{i}}
  $
\end{lemma}
\fi

\section{INTRODUCTION}\label{Sec: Introduction}
Cyber-physical System has application potential in various areas \cite{lee2016introduction}. 
The authors in \cite{rajkumar2010cyber} pointed out that the ideal CPS must operate dependably, safely, securely, efficiently and in real-time.

Security questions in CPSs are more challenging than traditional IT security because of the combination of temporal dynamics brought by the physical environment and the heterogeneous nature of the operation of CPSs \cite{khaitan2014design}. Failure of CPS is more complicated than random failures or well-defined uncertainty for which many results exist on reliability and robustness, since they may be caused by stealth malicious attacks. One of such powerful deception attacks, named \emph{false data injection attack} (FDIA), has shown ability to bypass \emph{bad data detector} (BDD), while compromising the integrity of observer and robust controller with sparse measurements corruption \cite{mo2010false, anubi2018robust}. Consequently, much research attention have been directed to develop appropriate protection schemes.


Active detection approaches have been considered \cite{weerakkody2019resilient}, where the defenders adjust the detection rules online in order to identify the attack scenarios.  There are some \emph{machine learning} algorithms being considered to localize the attacks, such as \emph{Gaussian process regression} \cite{anubi2019enhanced}, \emph{support vector machine}\cite{ozay2015machine}, \emph{markov graphs}\cite{sedghi2015statistical}, \emph{generative adversarial networks}\cite{mao2017least} and more. These \emph{machine learning} localization algorithms generate estimated support of attacked (or safe) nodes. This can then be used as a \emph{prior} information for a resilient estimation program.


Due to sparsity assumption of the attack vector $|\textsf{supp}(\mathbf{e})| \leq k$, the resilient estimation problem has been cast as a classical error correction problem \cite{anubi2018robust, anubi2019resilient}. Consider a linear observation model $\mathbf{y} = H\mathbf{x}+\mathbf{e}$, where $H \in \mathbb{R}^{N\times n}$ denotes an observation matrix, then the resilient estimation is formulated as $0$-norm minimization problem \cite{fawzi2014secure}. But it imposes a restriction of maximum attack percentage of $50\%$ for correct recovery of $\mathbf{x}$. Moreover, since $0$-norm minimization decoder is an NP-hard problem, an alternative $1$-norm minimization decoder has been considered in literature \cite{anubi2019enhanced, fawzi2014secure}, which can be solved by linear programming \cite{candes2005decoding}. The condition to bridge the two decoders is \emph{Restricted Isometry Property} (RIP) \cite{candes2006stable} that defines the sparse recoverability of observation matrix $H$. 

In order to guarantee correct state recoveries for larger percentage of attacks, \emph{prior} information has been considered for resilient estimation scheme in literature: \emph{Measurement Prior} \cite{anubi2019enhanced,anubi2020multi}, \emph{Support Prior} \cite{anubi2018robust} and \emph{State Prior} \cite{shinohara2019resilient}. 
In \cite{shinohara2019resilient}, the author considered the \emph{Prior} information of estimated states in three forms: \emph{sparsity information} of $\mathbf{x}_0$, \emph{$(\alpha, \overline{n}_0)$ sparsity information} where $\alpha$ replaces $0$ in sparsity definition, and \emph{side information} that is knowledge of the initial state from the physical attribution of the system and cannot be manipulated by malicious third parties. 
In \cite{anubi2019enhanced,anubi2020multi}, the author constructed a data-driven auxiliary model between system measurement and auxiliary state by trained \emph{Gaussian Process Regression} (GPR), the attacked measurements are visible to defender if they cannot be explained based on the \emph{measurement model prior} with high likelihood. 
In this paper, we consider \emph{support prior} which gives an estimated set of attack location, and is generated by any of the afore mentioned localization algorithms. However, there are two drawbacks, namely: uncertainty and training price. Thus, we propose a \emph{Pruning} method to improve the precision of the \emph{support prior} without training process. Then a weighted $1$-norm minimization scheme \cite{friedlander2011recovering} is given based on the resulting \emph{pruned support prior}. The pruning idea originates from \cite{anubi2018robust}. Moreover, due to the perfect localization precision of \emph{pruned support prior}, resilient \emph{Unscented Kalman Filter} (UKF) against FDIA was given in \cite{zheng2020attack} by performing UKF based on the pruned safe set.

The remainder of this paper is organized as follows. In \emph{Section \ref{Sec: Model_Development}}, we describe the concurrent models that will be used for the development in subsequent sections, including physical model of CPS, threat model and prior model. In \emph{Section \ref{Sec: Resilient_Observer_Design}}, we develop the pruning methods, and construct a weighted $\ell_1$ observer with \emph{pruned support prior}. A numerical simulation and an application simulation on IEEE-14 bus system show the proposed observer indeed enhance the system resilience in \emph{Section
\ref{Sec: Simulation}}. Finally, conclusion remarks follows in \emph{Section \ref{Sec: Conclusion}}.

\section{MODEL DEVELOPMENT}\label{Sec: Model_Development}
This section discusses the relevant models that will facilitate the development in subsequent sections. Specifically, we consider a physics-driven dynamical model for the physical side, a data-driven threat model for the cyber side and the prior information model.

\subsection{Dynamical Model}
Consider a linear model of CPS given by:
\begin{equation}\label{equ:system_model}
\begin{aligned}
    \mathbf{x}_{i+1} &= A\mathbf{x}_i \\
    \mathbf{y}_i &= C\mathbf{x}_i + \mathbf{e}_i,
    \end{aligned}
\end{equation}
where, $\mathbf{x}_i \in \mathbb{R}^n, \mathbf{y}_i \in \mathbb{R}^m$, with $m >n$, denote state vector and measurement vector at time $i$ respectively, $\mathbf{e}_i \in \mathscr{S}_k^m$ denotes the sparse attack vector, $A,C$ are system dynamic parameters. A control input may be included in the model above. However, since the control input is generally irrelevant to state estimation problems, we suppress in the model considered here. The following assumption is made regarding the CPS model above:
\begin{assumption}
The pair $(A,C)$ is full observable.
\end{assumption}
By iterating the system model \eqref{equ:system_model} $T$ time steps backwards, the $T$ horizon observation model is given by
\begin{equation}\label{equ:measurement_model}
    \mathbf{y}_T = H\mathbf{x}_{i-T+1}+\mathbf{e}_T,
\end{equation}
where $\mathbf{y}_T=[\mathbf{y}_i^{\top} \hspace{2mm} \mathbf{y}_{i-1}^{\top} \cdots \mathbf{y}_{i-T+1}^{\top}]^{\top} \in \mathbb{R}^{Tm}$ is a sequence of observation in the moving window $[i-T+1 \hspace{2mm} i]$, $\mathbf{x}_{i-T+1}\in \mathbb{R}^n$ is the state vector at time $i-T+1$, $\mathbf{e}_T = [\mathbf{e}_i^{\top} \hspace{2mm} \mathbf{e}_{i-1}^{\top} \cdots \mathbf{e}_{i-T+1}^{\top}]^{\top}$ is the sequence of attack vectors in the same moving window with $\mathbf{e}_i \in \mathscr{S}_{k}^m, \forall i \in [i-T+1 \hspace{2mm} i]$ and
$$ H = \begin{bmatrix} C A^{T-1} \\ \vdots \\ C A \\ C \end{bmatrix} =U \begin{bmatrix} \Sigma_1 \\0\end{bmatrix}V^{\top}
$$
where, $U = [U_1 \hspace{2mm} U_2]$, with $U_1 \in {\mathbb R}^{Tm \times n}, U_2 \in {\mathbb R}^{Tm \times Tm-n}$, is a matrix of left singular vectors of $H$, $V \in \mathbb{R}^{n \times n}$ is the corresponding matrix of right singular vectors, and $\Sigma_1 \in \mathbb{R}^{n \times n}$ is a diagonal matrix of the singular values of $H$, which are non-zero since $H$ is full rank.

\subsection{Threat Model}
We begin by defining a state decoder operation and a corresponding residual-based attack detector. Then we give a class of FDIA with guaranteed success against the defined decoder-detector pair. The following assumption is made concerning the attack vector:
\begin{assumption}
The attacker has knowledge of the system dynamics in \eqref{equ:system_model}.
\end{assumption}
\begin{definition}[Decoder]
Given a sequence of observation $\mathbf{y}_T \in \mathbb{R}^{Tm}$, a decoder for the measurement model in \eqref{equ:measurement_model} is a mapping of the form $\mathscr{D}:\mathbb{R}^{Tm} \mapsto \mathbb{R}^n$ given by
\begin{equation}\label{equ:decoder}
    \hat{\mathbf{x}} = \mathscr{D}(\mathbf{y}_T) =V\Sigma_1^{-1}\argmin_{\mathbf{z}}\|\mathbf{y}_T-U_1\mathbf{z}\|_{1}
\end{equation}
\end{definition}

\begin{definition}[Detector]
Given a positive threshold parameter $\epsilon$, a detector for the decoder $\mathscr{D}$ and the measurement model in \eqref{equ:measurement_model} is a binary classifier of the form
\begin{equation}\label{equ:detector}
    \mathscr{D}_{\epsilon}(\mathbf{y}_T) = \left\{
     \begin{array}{lr}
     1 &\quad\text{if}\hspace{0.2cm} \|\mathbf{y}_T - H\mathscr{D}(\mathbf{y}_T)\|_{1} > \epsilon \\
     0 &\quad\text{otherwise}
     \end{array}
     \right.
\end{equation}
\end{definition}

\begin{remark}
The positive class for $\mathscr{D}_{\epsilon}(\cdot)$ is regarded as the set of attacked (unsafe) measurements, while the negative class is regarded as the set of safe measurements. However, the negative class still contains unsafe measurements if the detector is compromised.
\end{remark}

\begin{definition}[Successful FDIA \cite{mo2010false}]\label{Def:Successful_FDIA}
Consider the CPS in \eqref{equ:system_model} and the corresponding measurement model \eqref{equ:measurement_model}, the attack sequence $\mathbf{e}_T\in\mathscr{S}^{Tm}_{km}$ is said to be $(\epsilon, \alpha)$-successful against the decoder-detector pair $\{\mathscr{D}, \mathscr{D}_{\epsilon}\}$ if
\begin{equation}\label{equ:def_success_FDIA}
     \|\mathbf{x}^{\star}-\mathscr{D}(\mathbf{y}_T)\|_2 \geq \alpha, \hspace{2mm} \mathscr{D}_{\epsilon}(\mathbf{y}_T)=0,
\end{equation}
where $\mathbf{y}_T = \mathbf{y}^{\star}_T+\mathbf{e}_T$ with $\mathbf{y}^*_T\in\mathbb{R}^{N}$ the true measurement vector, and $\mathbf{x}^{\star}$ is the true state vector.
\end{definition}
The following theorem gives a mechanism for constructing such $(\epsilon,\alpha)$-successful FDIA if the attack support is pre-determined.

\begin{theorem}\label{Thm:FDIA}
Given the support sequence $\mathscr{T} = \{\mathscr{T}_i \hspace{2mm} \mathscr{T}_{i-1} \cdots \mathscr{T}_{i-T+1}\}$ with $|\mathscr{T}_i|\leq k$. Let $\mathbf{z}_e$ be an optimal solution of the optimization program
\begin{equation}\label{equ:FDIA_optimization}
\begin{aligned}
    \textsf{Maxmize}&: \hspace{0.2cm} \|U_1 \mathbf{z}\|_{2}, \\
    \textsf{Subject to}&: \hspace{0.2cm} \| U_{1,\mathscr{T}^c}\mathbf{z}\|_{2} \leq \frac{\epsilon}{\sqrt{Tm-|\mathscr{T}|}}.
\end{aligned}
\end{equation}
If 
$
\|U_{1,\mathscr{T}^c}\|_2 < \frac{1}{2\sqrt{Tm-|\mathscr{T}|}},
$
  then the FDIA
\begin{equation}\label{equ:FDIA}
    \mathbf{e}_{\mathscr{T}} = (U_{1,\mathscr{T}}) \mathbf{z}_e, \hspace{2mm} \mathbf{e}_{\mathscr{T}^c} = \mathbf{0}
\end{equation} 
is $(\epsilon, \alpha)$-successful against the decoder-detector pair $\{\mathscr{D},\mathscr{D}_\epsilon\}$ for all 
$$ \alpha \leq \frac{\epsilon}{2\sqrt{Tm}\overline{\sigma}}\left(\frac{1}{\overline{\sigma}_{\mathscr{T}^c}\sqrt{Tm-|\mathscr{T}|}} - 2\right),
$$
where $\overline{\sigma}_{\mathscr{T}^c}$ is the largest singular value of $U_{1,\mathscr{T}^c}$ and $\overline{\sigma}$ is the biggest nonzero singular value of $H$.
\end{theorem}
\begin{proof}
Since 
$$\begin{aligned}
    \left\{\mathbf{z}\,\middle\vert\ \|\mathbf{z}\|_{2}\leq \frac{\epsilon}{\overline{\sigma}_{\mathscr{T}^c}\sqrt{Tm-|\mathscr{T}|}} \triangleq a\right\} \subset \\
\left\{\mathbf{z} \,\middle\vert\ \|U_{1,\mathscr{T}^c}\mathbf{z}\|_{2} \leq \frac{\epsilon}{\sqrt{Tm-|\mathscr{T}|}}\right\},
\end{aligned} 
$$
then
$$\begin{aligned}\|U_1 \mathbf{z}_e\|_{2} \geq \max_{\|\mathbf{z}\|_{2} \leq a} \| \mathbf{z}\|_{2} \geq \frac{\epsilon}{\overline{\sigma}_{\mathscr{T}^c}\sqrt{Tm-|\mathscr{T}|}}.
\end{aligned}$$
Also
$$
    \mathbf{y}_T=\mathbf{y}^{\star}_T+P\begin{bmatrix}U_{1,\mathscr{T}} \\ 0\end{bmatrix}\mathbf{z}_e,
$$
with an appropriate permutation matrix $P$ satisfying $U_1 = P\begin{bmatrix}U_{1,\mathscr{T}} \\ U_{1,\mathscr{T}^c}\end{bmatrix}$. Hence, by decoder in \eqref{equ:decoder},
\begin{equation}\label{equ:solution_decoder}
    \begin{aligned}
    \hat{\mathbf{x}} &= V \Sigma_1^{-1} \argmin_{\mathbf{z}}\left\|U_1(\mathbf{z}^{\star}-\mathbf{z}) + P\begin{bmatrix}U_{1,\mathscr{T}}\\ 0\end{bmatrix}\mathbf{z}_e\right\|_{1} \\
&= V\Sigma_1^{-1}(\mathbf{z}^{\star}-\mathbf{z}_e^{\perp}),
\end{aligned}
\end{equation}
where $\mathbf{z}^{\star}=\Sigma_1 V^{\top}\mathbf{x}^{\star}$ corresponds to the true state and $\mathbf{z}_e^{\perp}$ is the projection given by 
$$\mathbf{z}_e^{\perp} =\argmin_{\mathbf{z}}\left\|U_1\mathbf{z} + P\begin{bmatrix}U_{1,\mathscr{T}}\\ 0\end{bmatrix}\mathbf{z}_e\right\|_{1}.
$$
Then\footnote{Let $f(\mathbf{z}) = \left\|U_1\mathbf{z} - P\begin{bmatrix}0\\ U_{1,\mathscr{T}^c}\end{bmatrix}\mathbf{z}_e\right\|_{1}$ which is a convex function.The unique minimizer $\mathbf{z}_e^{\perp}$ satisfies $f(\mathbf{z}_e^{\perp}) \le f(\mathbf{z}_e^{\perp}-\mathbf{z}_e)$.},
$$\begin{aligned}\left\|U_1\mathbf{z}_e^{\perp} + P\begin{bmatrix}U_{1,\mathscr{T}}\\ 0\end{bmatrix}\mathbf{z}_e\right\|_{1} &\leq \left\|U_1\mathbf{z}_e^{\perp} - P\begin{bmatrix}0\\ U_{1,\mathscr{T}^c}\end{bmatrix}\mathbf{z}_e\right\|_1  \\
 \left\|U_{1,\mathscr{T}} \mathbf{z}_e\right\|_1-\|U_1\mathbf{z}_e^{\perp}\|_1  &\leq \|U_1\mathbf{z}_e^{\perp}\|_1+\left\|U_{1,\mathscr{T}^c} \mathbf{z}_e\right\|_1   \\
 2\|U_1\mathbf{z}_e^{\perp}\|_2 &\geq \left\|U_{1,\mathscr{T}} \mathbf{z}_e\right\|_1-\left\|U_{1,\mathscr{T}^c} \mathbf{z}_e\right\|_1 \\
\end{aligned}$$
Thus,
$$\begin{aligned}\|U_1\mathbf{z}_e^{\perp}\|_2 &\geq \frac{1}{2}\left(\left\|U_{1} \mathbf{z}_e\right\|_2-2\sqrt{Tm-|\mathscr{T}|}\left\|U_{1,\mathscr{T}^c} \mathbf{z}_e\right\|_2\right) \\
&\geq \frac{1}{2}\left(\frac{\epsilon}{\overline{\sigma}_{\mathscr{T}^c}\sqrt{Tm-|\mathscr{T}|}}-2\epsilon\right)
\end{aligned}$$
Next, according to the solution in \eqref{equ:solution_decoder}, it follows
$$\begin{aligned} \mathbf{x}^{\star} - \hat{\mathbf{x}} &= V \Sigma_1^{-1}\mathbf{z}_e^{\perp}\\
H(\mathbf{x}^{\star} - \hat{\mathbf{x}}) &= U_1 \mathbf{z}_e^{\perp}\\
\overline{\sigma} \|\mathbf{x}^{\star} - \hat{\mathbf{x}}\|_2 &\geq \frac{1}{\sqrt{Tm}}\|U_1 \mathbf{z}_e^{\perp}\|_1\\
\|\mathbf{x}^{\star} - \hat{\mathbf{x}}\|_2 &\geq \frac{1}{2\overline{\sigma}\sqrt{Tm}}\left(\frac{\epsilon}{\overline{\sigma}_{\mathscr{T}^c}\sqrt{Tm-|\mathscr{T}|}}-2\epsilon\right)
\end{aligned}$$
Since $
\|U_{1,\mathscr{T}^c}\|_2 < \frac{1}{2\sqrt{Tm-|\mathscr{T}|}},
$ the lower bound in the above inequality is positive. 
Moreover,
$$\begin{aligned}
\|\mathbf{y}_T &-H\hat{\mathbf{x}}\|_{1} =\left\|U_1\mathbf{z}^{\star}+P\begin{bmatrix}U_{1,\mathscr{T}}\\ 0\end{bmatrix}\mathbf{z}_e-H\hat{\mathbf{x}}\right\|_{1} \\
&=\left\|U_1\mathbf{z}^{\star}+P\begin{bmatrix}U_{1,\mathscr{T}}\\ 0\end{bmatrix}\mathbf{z}_e-U_1(\mathbf{z}^{\star}-\mathbf{z}_e^{\perp})\right\|_{1}\\
& = \left\|U_1\mathbf{z}_e^{\perp}+P\begin{bmatrix}U_{1,\mathscr{T}}\\ 0\end{bmatrix}\mathbf{z}_e\right\|_{1} \\
&\leq\sqrt{Tm-|\mathscr{T}|}\left\|U_{1,\mathscr{T}^c} \mathbf{z}_e\right\|_2 \leq \epsilon
\end{aligned}$$
\end{proof}

\subsection{Prior Model}
The \emph{prior} information considered in this paper is an uncertain estimate of the support of the attack vector. There are many Machine Learning algorithms for estimating the location of attacks in a CPS \cite{ozay2015machine,mao2017least}. We will refer to such algorithm as \emph{localization algorithm} (or \emph{localization oracle}), and the resulting support estimate as \emph{support prior}. 

Let $\mathscr{T}\subseteq\left\{1,2,\hdots, Tm\right\}$ be the actual support of the attacked nodes with the vector $\mathbf{q}\in\{0,\hspace{2mm}1\}^{Tm}$ the corresponding indicator
\begin{equation} \label{equ:support_indicator}
     \mathbf{q}_{i}=\left\{
     \begin{array}{lr}
     0 &\quad\text{if}\hspace{0.2cm} i \in \mathscr{T} \\
     1 &\quad\text{otherwise}
     \end{array}
     \right.
 \end{equation}
 
 
 Let $\hat{\mathscr{T}}$, be an estimate of $\mathscr{T}$, with the corresponding indicator $\hat{\mathbf{q}}\in\{0,1\}^{Tm}$ defined similarly to \eqref{equ:support_indicator}. Then an uncertainty model is defined as
  \begin{equation} \label{equ:uncertainty_model}
     \mathbf{q}_{i}=\epsilon_{i}\hat{\mathbf{q}}_{i}+(1-\epsilon_{i})(1-\hat{\mathbf{q}}_{i})
 \end{equation}
 where $\epsilon_i \sim \mathscr{B}(1,\mathbf{p}_i)$, with known $\mathbf{p}_i \in (0,1]$.
 
\begin{definition}[Positive Prediction Value, Precision, $\textsf{PPV}$]\label{Def:PPV_def}
Given the indicator vector estimate $\hat{\mathbf{q}}\in\{0,1\}^{Tm}$ of the unknown attack support indicator $\mathbf{q}\in\{0,1\}^{Tm}$, $\textsf{PPV}$ is the proportion of $\mathbf{q}$ that is correctly identified in $\hat{\mathbf{q}}$. It is given by
\begin{equation}\label{equ:PPV_def}
\textsf{PPV} =  \frac{\|\mathbf{q} \circ \hat{\mathbf{q}} \|_{0}}{\|\hat{\mathbf{q}}\|_{0}} 
\end{equation}
\end{definition}

\section{RESILIENT OBSERVER DESIGN} \label{Sec: Resilient_Observer_Design}

In this section, the pruning method is developed based on any uncertain estimate of \emph{support prior} from underlying \emph{machine learning} localization algorithm, then the generated \emph{pruned support prior} with precision guarantee is included in a \emph{weighted $1$-norm minimization} scheme.
\subsection{Prior Pruning}
Based on the knowledge of uncertainty of \emph{prior}, a \emph{pruned support prior} $\hat{\mathscr{T}}_{\eta}^c$ is generated in two steps; first, an estimated safe nodes support vector is returned by the below \emph{offline} optimization program, given a reliability level $\eta \in (0,1)$,
\begin{equation}
    \begin{aligned}
    \textsf{Maximize} &\hspace{2mm} |\hat{\mathscr{I}}| \\
    \textsf{Subject to}& \hspace{2mm} \prod\limits_{i\in \hat{\mathscr{I}}}\mathbf{p}_i \geq \eta\\
    &\hspace{2mm}\mathscr{I}\in\left\{1,2,\hdots,Tm\right\}.
    \end{aligned}
\end{equation}

Then a \emph{pruned support prior} is obtained \emph{online} through a robust extraction:
\begin{align}\label{equ:pruning_alg}
    \hat{\mathscr{T}}_\eta^c = \mathscr{I}\cap\hat{\mathscr{T}}^c.
\end{align}

\begin{lemma}\label{Lem:PPV_1}
Given a \emph{support prior} $\hat{\mathscr{T}}$ generated by any underlying localization algorithm with associated uncertainty model in \eqref{equ:uncertainty_model}, if a \emph{pruned support prior} $\hat{\mathscr{T}}_{\eta}^c$ is generated by \eqref{equ:pruning_alg}, then it satisfies
\begin{equation}
    \textsf{Pr}\{\textsf{PPV}_{\eta} = 1\} \geq \eta,
\end{equation}
where $\textsf{PPV}_{\eta}$ is the precision of $\hat{\mathscr{T}}_{\eta}^c$ defined by \eqref{equ:PPV_def}.
\end{lemma}
\begin{proof}
From \eqref{equ:uncertainty_model}, it follows 
$\mathbf{q}_i \hat{\mathbf{q}}_i = \epsilon_i \hat{\mathbf{q}}_i,
$
which implies
$$
\textsf{PPV} =  \frac{\|\mathbf{q} \circ \hat{\mathbf{q}} \|_{0}}{\|\hat{\mathbf{q}}\|_{0}} = \frac{\sum_{i=1}^N \mathbf{q}_i \hat{\mathbf{q}}_i}{\sum_{i=1}^N \hat{\mathbf{q}}_i}
=\frac{1}{|\hat{\mathscr{T}}^c|}\sum_{i=1}^N \epsilon_i \hat{\mathbf{q}}_i = \frac{1}{|\hat{\mathscr{T}}^c|}\sum_{i \in \hat{\mathscr{T}}^c}\epsilon_i
$$
Similarly, 
$$ \textsf{PPV}_{\eta} = \frac{1}{|\hat{\mathscr{T}}^c_{\eta}|}\sum_{i \in \hat{\mathscr{T}}^c_{\eta}}\epsilon_i.
$$
Then, since $\hat{\mathscr{T}}_\eta^c\subseteq\mathscr{I}$ and $0<\mathbf{p}_i\le1$,
$$\textsf{Pr}\{\textsf{PPV}_{\eta} = 1\} = \textsf{Pr}\left\{\sum_{i \in \hat{\mathscr{T}}^c_{\eta}}\epsilon_i =|\hat{\mathscr{T}}^c_{\eta}|\right\} = \prod_{i\in \hat{\mathscr{T}}_{\eta}^c}\mathbf{p}_i \geq \eta
$$
\end{proof}

\subsection{Weighted $\ell_1$ Observer}
Consider the CPS in \eqref{equ:system_model}. Given a \emph{pruned support prior} $\hat{\mathscr{T}}^{c}_{\eta}$ in \eqref{equ:pruning_alg}, the weighted $\ell_1$ observer is given by the following moving horizon estimation problem:
\begin{equation}\label{equ:MHE_weighted_l1_minimization}
    \begin{aligned}
    \textsf{Minimize} &\hspace{2mm} \sum_{p=i-T+1}^i \|\mathbf{y}_p-C\mathbf{z}_p\|_{1,w}\\
    \textsf{Subject to}& \hspace{2mm} \mathbf{z}_{p+1} -A\mathbf{z}_p =0, \\
    &\hspace{10mm} p = i-T+1, \cdots,i-1 \\
    &\hspace{2mm} w = \left\{\begin{array}{ll}
1, & j \in \hat{\mathscr{T}}^{c}_{\eta} \\
\omega, & j \in \hat{\mathscr{T}}_{\eta}
\end{array}\right.
    \end{aligned}
\end{equation}
where, $0\leq \omega \leq 1$, $\|\mathbf{z}\|_{1, \mathrm{w}} \triangleq \sum_j \mathrm{w}_j |\mathbf{z}_j|$ is the weighted $1$-norm.

\begin{remark}
Consider the corresponding measurement model \eqref{equ:measurement_model}, then an equivalent form of \eqref{equ:MHE_weighted_l1_minimization} is given by
 \begin{equation}\label{equ:weighted_l1_minimization}
\begin{aligned}
\Minimize_{\mathbf{z}\in \mathbb{R}^n}\hspace{1mm} \|\mathbf{y}_T - H \mathbf{z}\|_{1, \mathrm{w}}, \hspace{1mm} \text{with} \hspace{1mm}w = \left\{\begin{array}{ll}
1, & j \in \hat{\mathscr{T}}^{c}_{\eta} \\
\omega, & j \in \hat{\mathscr{T}}_{\eta}
\end{array}\right.
\end{aligned}    
\end{equation}
\end{remark}

\begin{theorem}\label{Thm:recovery_Withpruning}
Consider the CPS measurement model in \eqref{equ:measurement_model}, such that $U_2^{\top}$ satisfies the RIP condition
\begin{align}\label{equ: RIP}
   (1-\delta_{km})\|\mathbf{h}\|_{2}^2 \leq \|U_{2}^{\top}\mathbf{h}\|_{2}^2 \leq (1+\delta_{km})\|\mathbf{h}\|_{2}^2 
 \end{align}
for all sparse vector $\mathbf{h} \in \mathscr{S}^{Tm}_{km}$.
Given a \emph{support prior} $\hat{\mathscr{T}}^c$ generated by an underlying localization algorithm with associated uncertainty model in \eqref{equ:uncertainty_model}.
Let $\hat{\mathscr{T}}_{\eta}^c$ be the \emph{pruned support prior} given by \eqref{equ:pruning_alg} with a parameter $\eta \in (0,1)$. If there exists an integer $a\geq \max\left\{\rho-1,1\right\}$, where $\rho = \frac{|\hat{\mathscr{T}}_{\eta}|}{km}$, such that 
\begin{equation}
    \delta_{akm}+C\delta_{(a+1)km} \leq C-1
\end{equation} 
then, with a probability of at least $\eta$, the estimation error due to the resilient observer in \eqref{equ:weighted_l1_minimization} can be upper bounded as
\begin{equation} \label{equ:error_bound3}
    \left\|\hat{\mathbf{x}}-\mathbf{x}^*\right\|_{2} \leq 
   \frac{C_{1}}{\underline{\sigma}\sqrt{km}}\left(\omega\sigma_{km}(\mathbf{e})+(1-\omega)\|\mathbf{e}_{\hat{\mathscr{T}}_{\eta}^c} \|_{1}\right),
\end{equation}
where $\underline{\sigma}$ is the smallest singular value of $H$, and 
$$\begin{aligned} C_1 &= \frac{2a^{-\frac{1}{2}}(\sqrt{1-\delta_{(a+1)km}}+\sqrt{1+\delta_{a km}})}{\sqrt{1-\delta_{(a+1)km}}-\frac{1}{C}\sqrt{1+\delta_{akm}}}, \\  C &= \frac{a}{(\omega+(1-\omega)\sqrt{\rho -1})^2}.
\end{aligned}$$
\end{theorem}

\begin{proof}
To avoid repetition, we state upfront that all claims made in this rest of proof holds with a probability of at least $\eta$. Based on \emph{Lemma \ref{Lem:PPV_1}}, we have $\hat{\mathscr{T}}^{c}_{\eta} \subset \mathscr{T}^c$, $\mathscr{T}\subset \hat{\mathscr{T}}_{\eta}$ under this probability.

Let $\hat{\mathbf{x}}$ be a solution to \eqref{equ:weighted_l1_minimization} with $H\hat{\mathbf{x}} = H\mathbf{x}^*-\mathbf{h}$, then the corresponding minimizer $\hat{\mathbf{e}} = \mathbf{e}+\mathbf{h}$, thus $\|\hat{\mathbf{e}}\|_{1,\omega} \leq \|\mathbf{e}\|_{1,\omega}$. Following the definition of weight $w$ in \eqref{equ:weighted_l1_minimization}, 
$$\omega \|\mathbf{e}_{\hat{\mathscr{T}}_{\eta}}+\mathbf{h}_{\hat{\mathscr{T}}_{\eta}} \|_{1} + \|\mathbf{e}_{\hat{\mathscr{T}}^{c}_{\eta}}+\mathbf{h}_{\hat{\mathscr{T}}^{c}_{\eta}}\|_{1} \leq \omega \|\mathbf{e}_{\hat{\mathscr{T}}_{\eta}}\|_{1}
+ \|\mathbf{e}_{\hat{\mathscr{T}}^c_{\eta}}\|_{1}
$$
Since $1$-norm is decomposable for disjoint sets, such for $\mathscr{T}$ and $\mathscr{T}^c$, and $\hat{\mathscr{T}}^{c}_{\eta} \cap\mathscr{T} =\emptyset $, $\hat{\mathscr{T}}^{c}_{\eta} \cap \mathscr{T}^c =\hat{\mathscr{T}}^{c}_{\eta}$, $\mathscr{T}\cap \hat{\mathscr{T}}_{\eta}=\mathscr{T}$, it follows
$$\begin{aligned} \omega \|\mathbf{e}_{\mathscr{T}}+\mathbf{h}_{\mathscr{T}}\|_{1} +\omega\|\mathbf{e}_{\hat{\mathscr{T}}_{\eta} \cap \mathscr{T}^c}+\mathbf{h}_{\hat{\mathscr{T}}_{\eta} \cap \mathscr{T}^c}\|_{1} + \|\mathbf{e}_{\hat{\mathscr{T}}^{c}_{\eta}}+\mathbf{h}_{\hat{\mathscr{T}}^{c}_{\eta}}\|_{1} \\
\leq \omega\|\mathbf{e}_{\mathscr{T}}\|_{1} +\omega \|\mathbf{e}_{\hat{\mathscr{T}}_{\eta}\cap \mathscr{T}^c}\|_{1} +\|\mathbf{e}_{\hat{\mathscr{T}}^c_{\eta}}\|_{1}
\end{aligned}$$
By triangle inequality, it follows
$$ \omega\|\mathbf{h}_{\hat{\mathscr{T}}_{\eta} \cap \mathscr{T}^c}\|_{1}+\|\mathbf{h}_{\hat{\mathscr{T}}_{\eta}^c}\|_{1} \leq \omega \|\mathbf{h}_{\mathscr{T}}\|_{1} +2\left(\|\mathbf{e}_{\hat{\mathscr{T}}_{\eta}^c}\|_{1}+\omega\|\mathbf{e}_{\hat{\mathscr{T}}_{\eta} \cap \mathscr{T}^c}\|_{1}\right)
$$
Adding and subtracting $\omega\|\mathbf{h}_{\hat{\mathscr{T}}_{\eta}^c \cap \mathscr{T}^c}\|_{1} (= \omega\|\mathbf{h}_{\hat{\mathscr{T}}_{\eta}^c}\|_{1})$ on LHS, $2\omega\|\mathbf{e}_{\hat{\mathscr{T}}_{\eta}^c \cap \mathscr{T}^c}\|_{1} (= 2\omega\|\mathbf{e}_{\hat{\mathscr{T}}_{\eta}^c}\|_{1})$ on RHS, it yields
$$\omega\|\mathbf{h}_{\mathscr{T}^c}\|_{1}+(1-\omega)\|\mathbf{h}_{\hat{\mathscr{T}}_{\eta}^c}\|_{1} \leq  \omega\|\mathbf{h}_{\mathscr{T}}\|_{1} + \beta
$$
where, $\beta =2\left(\omega\|\mathbf{e}_{\mathscr{T}^c}\|_{1}+(1-\omega)\|\mathbf{e}_{\hat{\mathscr{T}}_{\eta}^c}\|_{1}\right)$.
And since $\|\mathbf{h}_{\mathscr{T}^c}\|_{1} = \omega\|\mathbf{h}_{\mathscr{T}^c}\|_{1}+(1-\omega)\|\mathbf{h}_{\hat{\mathscr{T}}_{\eta}^c}\|_{1}+ (1-\omega)\|\mathbf{h}_{\hat{\mathscr{T}}_{\eta} \cap \mathscr{T}^c}\|_{1}$, it yields
\begin{equation}\label{equ2}
\|\mathbf{h}_{\mathscr{T}^c}\|_{1} \leq \omega\|\mathbf{h}_{\mathscr{T}}\|_{1} + (1-\omega)\|\mathbf{h}_{\hat{\mathscr{T}}_{\eta} \cap \mathscr{T}^c}\|_{1}+ \beta
\end{equation}

Next, sort the coefficients of $\mathbf{h}_{\mathscr{T}^c}$ in descending order, and let $\mathscr{T}_j, j\in \{1,2,\cdots\}$ denote $j$th support in $\mathbf{h}_{\mathscr{T}^c}$ with size $akm \in \mathbb{Z}$, where $a>1$. Then $\|\mathbf{h}_{\mathscr{T}_{j}}\|_{2} \leq (akm)^{-1/2}\|\mathbf{h}_{\mathscr{T}_{j}}\|_{\infty}  \leq (akm)^{-1/2} \|\mathbf{h}_{\mathscr{T}_{j-1}}\|_{1}$, let $\mathscr{T}_0 = \mathscr{T} \cup \mathscr{T}_1$, we have
\begin{equation}\label{equ3}
\begin{aligned}
    \|\mathbf{h}_{\mathscr{T}_0^c}\|_{2} &\leq \sum_{j\geq 2}\|\mathbf{h}_{\mathscr{T}_j}\|_{2}\\ &\leq (akm)^{-1/2}\sum_{j\geq 1}\|\mathbf{h}_{\mathscr{T}_j}\|_{1} \\
    &\le  (akm)^{-1/2}\|\mathbf{h}_{\mathscr{T}^c}\|_{1}
    \end{aligned}
\end{equation} 
Note that $|\hat{\mathscr{T}}_{\eta} \cap \mathscr{T}^c|=(\rho-1)km $ since $\textsf{PPV}_{\eta}=1$, and $|\mathscr{T}| =km$, then $\|\mathbf{h}_{\hat{\mathscr{T}}_{\eta} \cap \mathscr{T}^c}\|_{1} \leq \sqrt{(\rho-1)km}\|\mathbf{h}_{\hat{\mathscr{T}}_{\eta} \cap \mathscr{T}^c}\|_{2}$ , $\|\mathbf{h}_{\mathscr{T}}\|_{1} \leq \sqrt{km}\|\mathbf{h}_{\mathscr{T}}\|_{2} \leq \sqrt{km}\|\mathbf{h}_{\mathscr{T}_0}\|_{2}$. Since $a \geq \rho-1$, we obtain $|\hat{\mathscr{T}}_{\eta} \cap \mathscr{T}^c|=(\rho-1)km \leq akm =|\mathscr{T}_1|$, 
then $\|\mathbf{h}_{\hat{\mathscr{T}}_{\eta} \cap \mathscr{T}^c}\|_{2} = \|\mathbf{h}_{\mathscr{T} \cup (\hat{\mathscr{T}}_{\eta} \cap \mathscr{T}^c)}\|_{2} \leq \|\mathbf{h}_{\mathscr{T}_0}\|_{2}$.
Then combine with \eqref{equ2} and \eqref{equ3},
\begin{equation}\label{equ4}
    \|\mathbf{h}_{\mathscr{T}_0^c}\|_{2} \leq  \frac{\omega+(1-\omega)\sqrt{\rho-1}}{\sqrt{a}}\|\mathbf{h}_{\mathscr{T}_0}\|_{2} +\frac{\beta}{\sqrt{akm}}
\end{equation}

Since $ \|U_2^{\top}\mathbf{h}\|_{2}=\|(U_2^{\top}\mathbf{e}^{\star}-U_2^{\top}\mathbf{y}_T)-(U_2^{\top}\mathbf{e}-U_2^{\top}\mathbf{y}_T)\|_{2} =0$, it follows, based on triangle inequality and RIP condition,
$$\begin{aligned} \sqrt{1-\delta_{(a+1)km}}\|\mathbf{h}_{\mathscr{T}_0}\|_{2} \leq \|U_2^{\top}\mathbf{h}_{\mathscr{T}_0}\|_{2} \leq \|U_2^{\top}\mathbf{h}_{\mathscr{T}_0^c}\|_{2},
\end{aligned}$$
and
$$\begin{aligned} \|U_2^{\top}\mathbf{h}_{\mathscr{T}_0^c}\|_{2} \leq 
\sum_{j\geq 2} \|U_2^{\top}\mathbf{h}_{\mathscr{T}_j}\|_{2} 
\leq  \sqrt{1+\delta_{akm}}\sum_{j\geq 2} \|\mathbf{h}_{\mathscr{T}_j}\|_{2}  \\ \leq \frac{\sqrt{1+\delta_{akm}}}{\sqrt{akm}}\sum_{j\geq 1} \|\mathbf{h}_{\mathscr{T}_j}\|_{1} = \frac{\sqrt{1+\delta_{akm}}}{\sqrt{akm}} \|\mathbf{h}_{\mathscr{T}^c}\|_{1},
\end{aligned}$$
then combining with \eqref{equ2} yields
$$\begin{aligned} \sqrt{1-\delta_{(a+1)km}}\|\mathbf{h}_{\mathscr{T}_0}\|_{2} \leq  \omega\frac{\sqrt{1+\delta_{akm}}}{\sqrt{akm}} \|\mathbf{h}_{\mathscr{T}}\|_{1} + \\ (1-\omega)\frac{\sqrt{1+\delta_{akm}}}{\sqrt{akm}} \|\mathbf{h}_{\hat{\mathscr{T}}_{\eta} \cap \mathscr{T}^c}\|_{1} +\frac{\sqrt{1+\delta_{akm}}}{\sqrt{akm}}\beta.
\end{aligned}
$$
Notice, we have $\|\mathbf{h}_{\mathscr{T}}\|_{1}\leq \sqrt{km}\|\mathbf{h}_{\mathscr{T}_0}\|_{2} $ and $\|\mathbf{h}_{\hat{\mathscr{T}}_{\eta} \cap \mathscr{T}^c}\|_{1} \leq \sqrt{(\rho-1)km}\|\mathbf{h}_{\mathscr{T}_0}\|_{2}$, then manipulating the above inequality yields
$$ \|\mathbf{h}_{\mathscr{T}_0}\|_{2} \leq \frac{\frac{\sqrt{1+\delta_{akm}}}{\sqrt{akm}} \beta}{\sqrt{1-\delta_{(a+1)km}}-\frac{\omega+(1-\omega)\sqrt{\rho-1}}{\sqrt{a}}\sqrt{1+\delta_{akm}}}
$$

Since $\|\mathbf{h}\|_{2}\leq \|\mathbf{h}_{\mathscr{T}_0}\|_{2} +\|\mathbf{h}_{\mathscr{T}_0^c}\|_{2} $, combining the above inequality with \eqref{equ4} yields
$$ \|\mathbf{h}\|_{2}\leq \frac{2\frac{\sqrt{1-\delta_{(a+1)km}}+\sqrt{1+\delta_{akm}}}{\sqrt{akm}}\left(\omega\|\mathbf{e}_{\mathscr{T}^c}\|_{1}+(1-\omega)\|\mathbf{e}_{\hat{\mathscr{T}}_{\eta}^c}\|_{1}\right)}{\sqrt{1-\delta_{(a+1)km}}-\frac{\omega+(1-\omega)\sqrt{\rho-1}}{\sqrt{a}}\sqrt{1+\delta_{akm}}}
$$
where $\|\mathbf{e}_{\mathscr{T}^c}\|_{1} = \sigma_{km}(\mathbf{e})$ is best $k$ sparse approximation error of $\mathbf{e}$ defined in \eqref{equ:best_k_sparse_estimation_error}, and the estimation error $\|\hat{\mathbf{x}}-\mathbf{x}^*\|_{2} \leq \underline{\sigma}^{-1}\|\mathbf{h}\|_{2}$. Moreover, in order to ensure that the denominator is positive, the following condition is imposed:
$$ \delta_{akm} + C\delta_{(a+1)km} < C-1,
$$
where $C = \frac{a}{(\omega+(1-\omega)\sqrt{\rho-1})^2}$.
\end{proof}

\section{NUMERICAL SIMULATION}\label{Sec: Simulation}
In this section, a numerical simulation\footnote{we make the numerical simulation be open source at \url{https://github.com/ZYblend/Resilient-Pruning-Observer}} and an application simulation on IEEE-14 bus system are presented. 

\subsection{Numerical Simulation}
In this simulation, we compare the resiliency of three estimation scheme: $\ell_1$ observer without prior, weighted $\ell_1$ observer with prior generated an underlying attack localization algorithm with a true rate of $T_r=0.6$, and weighted $\ell_1$ observer with the pruned prior.
The system dimension is set as $m=20,n=10$, and then a full observable system is generated with random pair $(A,C)$ of independent Gaussian entries \cite{candes2005decoding}. For different attack percentage $P_A$, the FDIA is designed by \eqref{equ:FDIA} on random support $\mathscr{T}$.
By defining "success" as that the estimation error is less than $0.1\%$ of the real state, the success percentage is calculated from $1000$ trials. In Figure~\ref{Sucess_Attack_Tr60_m20_n10}, a performance comparison is presented for varying attack percentages. 
As proved in literature \cite{fawzi2014secure}, \emph{$\ell_1$ observer without prior} cannot work when attack percentage is larger than $1/2$. The prior information can improve the resiliency of the estimator, but the improvement is limited because the precision of the prior information is not enough. By including pruning method, the resiliency is improved a lot.

\begin{figure}[h]
\begin{center}
\includegraphics[scale = 0.35]{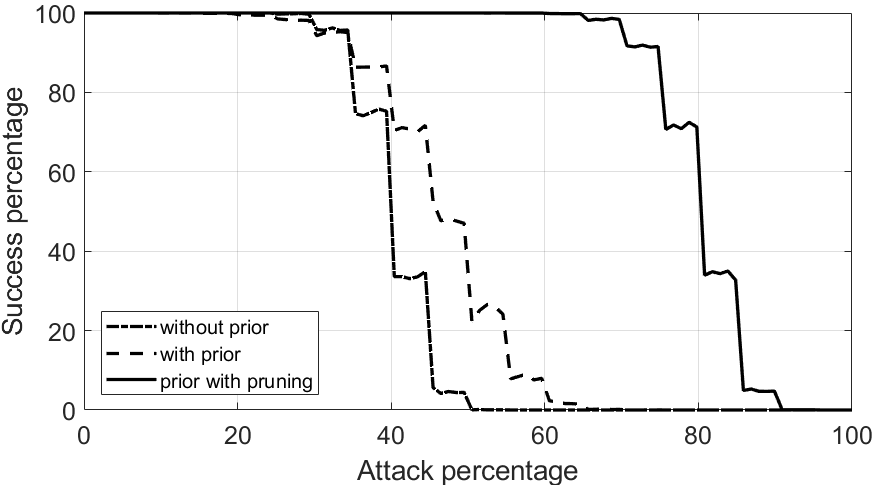}
\end{center}
\caption{A comparison of estimation performance under different attack percentage $P_A \in [0,1]$ between \emph{$\ell_1$ Observer without Prior}, \emph{Weighted $\ell_1$ Observer with Prior}, and \emph{Weighted $\ell_1$ Observer with Prior Pruning}.}
\label{Sucess_Attack_Tr60_m20_n10} 
\end{figure}

\subsection{Application Example}
In this subsection, an additional simulation is carried out for a more realistic IEEE Bus 14 system with a similar setup in \cite{anubi2020multi}. Here, $30\%$ of the measurement nodes are attacked. A comparison of the resulting phase angles estimation for an \emph{$\ell_1$ observer without prior} and the proposed \emph{weighted $\ell_1$ observer with prior pruning} is shown in Figure~\ref{angle_estimation}. Moreover, two comparison metrics for the estimation errors are given in Table~\ref{table_of_error_metric}. The first one is the \emph{root-mean-square} (RMS) value of error, the second one is the \emph{maximum absolute value} of the errors. Since the attack percentage is small, \emph{$\ell_1$ observer without prior} works well. However, there are still notable spikes that may induce closed loop instability. As shown in the Figure, the proposed \emph{weighted $\ell_1$ observer with prior pruning} completely eliminates the spikes!
\begin{table}[h]
\centering
 \caption{Error Metric Values\label{table_of_error_metric}}
 \begin{tabular}{||c | c | c | c | c | c | c||}
\hline
 &\multicolumn{3}{|c|}{RMS Metric}&\multicolumn{3}{|c||}{Max. Ans. Metric} \\
 \hline\hline
 & \textbf{LO} & \textbf{L1O} & \textbf{WL1P} & \textbf{LO} & \textbf{L1O} & \textbf{WL1P}\\
 \hline\hline
 $\delta_1$ & $2.5359$ & $0.0002$ & $0.00004$ & $5.7480$ & $0.0034$ & $0.0005$\\ 

$\delta_2$ & $2.3917$ & $0.0002$ & $0.0001$ & $5.7000$ & $0.0016$ & $0.0016$ \\

 $\delta_3$ & $2.6353$ & $0.0012$ & $0.0001$ & $7.5232$ & $0.0215$ & $0.0013$ \\

 $\delta_4$ & $2.3685$ & $0.0006$ & $0.0005$ & $5.7236$ & $0.0063$ & $0.0052$ \\
 
 $\delta_5$ & $2.6638$ & $0.0007$ & $0.0001$ & $7.8757$ & $0.0085$ & $0.0016$ \\  
 \hline \hline
 \multicolumn{7}{||c||}{\textbf{LO}: Luenberger observer, \textbf{L1O}: $\ell_1$ observer without prior} \\
\multicolumn{7}{||c||}{\textbf{WL1P}: Weighted $\ell_1$ observer with prior pruning} \\
\hline
\end{tabular}
\end{table}

\begin{figure}[t]
\begin{center}
\includegraphics[scale = 0.3]{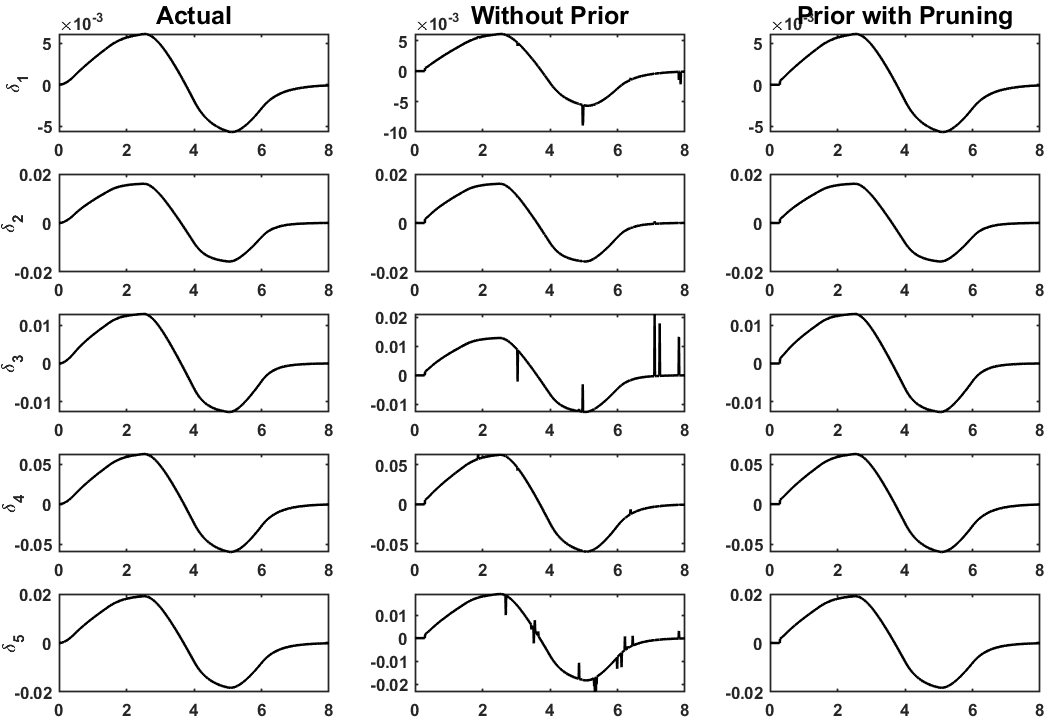}
\end{center}
\caption{Angle estimation for IEEE-14 bus system under FDIA, "without Prior" means \emph{$\ell_1$ Observer without Prior}, "Prior with Pruning" means \emph{Weighted $\ell_1$ Observer with Prior Pruning}.}
\label{angle_estimation} 
\end{figure}

\section{CONCLUSION}\label{Sec: Conclusion}
This paper proposed a weighted $\ell_1$ observer with prior pruning scheme against FDIAs. The pruning method gives a method to improve localization precision of any underlying localization algorithm without training effort. Moreover, the weighted $\ell_1$ observer with prior pruning is capable of coping with high-percentage of attacks among measurement nodes, which relaxes the transitional restriction on the maximum attack percentage for resilient $\ell_1$ observer, thereby improve the resiliency of systems.







\medskip
\printbibliography

\end{document}